\documentclass[a4paper]{amsart}

\usepackage{anysize}

\usepackage{amssymb,latexsym}
\usepackage{amsmath,amsthm}
\usepackage{amsfonts,mathrsfs}
\usepackage{mathtools}

\usepackage{halloweenmath}

\usepackage{tikz-cd}

\usepackage{todonotes}

\usepackage[all]{xy}
\usepackage{graphicx}

\usepackage{xcolor,url}
\usepackage{hyperref}
\hypersetup{
    colorlinks,
    citecolor=blue,
    filecolor=blue,
    linkcolor=black,
    urlcolor=blue
}

\usepackage{enumitem}

\usepackage{listings}
\usepackage{color}

\definecolor{dkgreen}{rgb}{0,0.6,0}
\definecolor{gray}{rgb}{0.5,0.5,0.5}
\definecolor{mauve}{rgb}{0.58,0,0.82}

\usepackage{forest}

\usepackage{tabularray}
\UseTblrLibrary{diagbox}

\theoremstyle{plain}
\newtheorem{theorem}{Theorem}[section]

\newtheorem{corollary}[theorem]{Corollary}
\newtheorem{lemma}[theorem]{Lemma}

\newtheorem{fact}[theorem]{Fact}
\newtheorem{conjecture}[theorem]{Conjecture}
\newtheorem*{theorem*}{Theorem}
\newtheorem*{corollary*}{Corollary}
\newtheorem*{conjecture*}{Conjecture}

\newtheorem*{generalburnsideproblem}{\textbf{The General Burnside Problem}}
\newtheorem*{burnsideproblem}{\textbf{The Burnside Problem}}
\newtheorem*{restrictedburnsideproblem}{\textbf{The Restricted Burnside Problem}}

\newtheorem*{generalengelgroupproblem}{\textbf{The General Local Nilpotency Problem for Engel Groups}}
\newtheorem*{engelgroupproblem}{\textbf{The Local Nilpotency Problem for Engel Groups}}

\newtheorem*{engelLAproblem}{\textbf{The Local Nilpotency Problem for Engel Lie algebras}}

\newtheorem*{generalkuroshproblem}{\textbf{The General Kurosh-Levitzki Problem}}
\newtheorem*{kuroshproblem}{\textbf{The Kurosh-Levitzki Problem}}

\theoremstyle{definition}

\theoremstyle{remark}
\newtheorem{remark}[theorem]{Remark}
\newtheorem{claim}{Claim}

\newcommand\N{\mathbb{N}}

\newcommand\F{\mathbb{F}}

\newcommand\LL{\mathscr{L}}

\newcommand{\End}{\mathrm{End}}

\def\seq{\subseteq}

\def\rteq{\trianglerighteq}

\newcommand{\set}[1]{\{ {#1} \}}

\DeclareMathOperator{\Aut}{Aut}

\DeclareMathOperator{\ad}{ad}
\DeclareMathOperator{\Span}{Span}

\definecolor{airforceblue}{rgb}{0.36, 0.54, 0.66}

\def\Ind{\setbox0=\hbox{$x$}\kern\wd0\hbox to 0pt{\hss$\mid$\hss}
\lower.9\ht0\hbox to 0pt{\hss$\smile$\hss}\kern\wd0}
\def\Notind{\setbox0=\hbox{$x$}\kern\wd0\hbox to 0pt{\mathchardef
\nn=12854\hss$\nn$\kern1.4\wd0\hss}\hbox to
0pt{\hss$\mid$\hss}\lower.9\ht0 \hbox to 0pt{\hss$\smile$\hss}\kern\wd0}

\def\indi#1{\mathop{\ \ \hbox to 0ex{\hss$\vert^{\hbox to 0ex{$\scriptstyle#1$\hss}}$\hss}
\lower1ex\hbox to 0ex{\hss$\smile$\hss}\ \ }}

\def\nindi#1{\mathop{\ \ \hbox to 0ex{\hss$\!\not{\vert}^{\hbox to 0ex{$\scriptstyle\,#1$\hss}}$\hss}
\lower1ex\hbox to 0ex{\hss$\smile$\hss}\ \ }}

\renewcommand{\models}{\vDash}

\begin{document}

\title{Wilson conjecture for omega-categorical Lie algebras, the case $3$-Engel characteristic $5$}

\author[C. d'Elb\'{e}e]{Christian d\textquoteright Elb\'ee$^\dagger$}
\address{School of Mathematics, University of Leeds\\
Office 10.17f LS2 9JT, Leeds}
\email{C.M.B.J.dElbee@leeds.ac.uk}
\urladdr{\href{http://choum.net/\textasciitilde chris/page\textunderscore perso/}{http://choum.net/\textasciitilde chris/page\textunderscore perso/}}

\thanks{
\begin{minipage}{0.8\textwidth}
 The author is fully supported by the UKRI Horizon Europe Guarantee Scheme, grant no EP/Y027833/1.
\end{minipage}%
\begin{minipage}{0.2\textwidth}
\begin{center}
    \includegraphics[scale=.04]{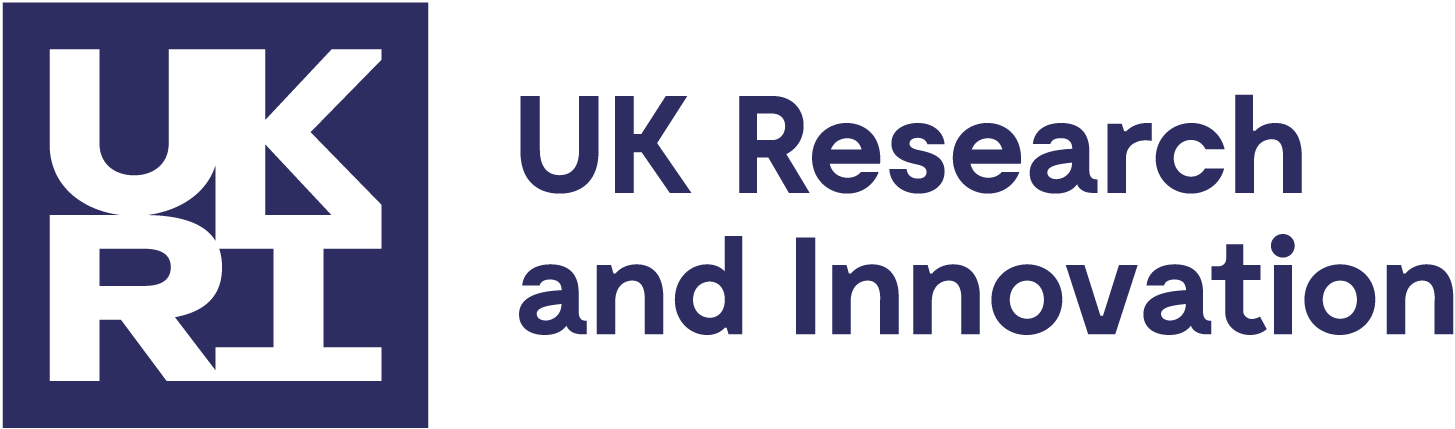}
\end{center}
\end{minipage}
}

\date{\today}

\begin{abstract}
We prove a version of the Wilson conjecture for $\omega$-categorical $3$-Engel Lie algebras over a field of characteristic $5$: every $\omega$-categorical Lie algebra over $\F_5$ which satisfies the identity $[x,y^3] = 0$ is nilpotent. We also include an extended introduction to Wilson's conjecture: \textit{every $\omega$-categorical locally nilpotent $p$-group is nilpotent}, and present variants of this conjecture and connections to local/global nilpotency problems (Burnside, Kurosh-Levitzki, Engel groups).
No particular knowledge of model theory is assumed except basic notions of formulas and definable sets.
\end{abstract}

\maketitle

\section{Introduction} 

A structure $M$ (group, Lie algebra, associative algebra, graph, etc) is \textit{$\omega$-categorical} if there is a unique countable model of its first-order theory up to isomorphism. By a classical result of Engeler, Ryll-Nardzewski, Svenonius, this notion has a dynamical definition: $M$ is $\omega$-categorical if and only if for each $n\geq 1$ there are finitely many orbits in the componentwise action of $\Aut(M)$ on the cartesian product $M^n$. All $\omega$-categorical structures considered here will be countable. The notion of $\omega$-categoricity captures \textit{highly symmetric} objects.

In a foundational paper \cite{wilson1981}, Wilson starts a systematic study of $\omega$-categorical groups and states the following conjecture:
\begin{conjecture}[Wilson, 1981]\label{conj:wilson}
    Every locally nilpotent $\omega$-categorical $p$-group is nilpotent.
\end{conjecture}

Under the conjecture, in every $\omega$-categorical group $G$ there is a finite characteristic series $G = G_1\rteq G_2\rteq \ldots \rteq G_n = 1$ such that each quotient $G_i/G_{i+1}$ is either an elementary abelian group or a boolean power of a finite simple group, see \cite{Apps83A}. A positive answer to Wilson's conjecture would also imply that every locally nilpotent $\omega$-categorical group is nilpotent. 

The conjecture has been solved under extra model-theoretic assumptions but seems overall out of reach for now. A few general structural results were obtained in the 80's by Rosenstein, Cherlin, Wilson, Apps \cite{Rosenstein1973, wilson1981, Apps83A,Apps83B, Apps83C} then later by Macpherson \cite{Macpherson1988ubiquitousomcatgroups} and Archer and Macpherson \cite{ArcherMacpherson1997}.

Model theorists have been quite successful in analyzing the structure of \textit{tame} $\omega$-categorical groups, starting with Felgner \cite{Felgner1978} who described the stable ones, followed then by a trend of results by Baur, Cherlin, Macintyre \cite{BCM79}, Macpherson (e.g. Wilson's conjecture holds for NSOP groups \cite{Macpherson1988ubiquitousomcatgroups}), Krupinski, Wagner, Evans, Derakhshan, Dobrowolski \cite{DerakhshanWagner1997, EvansWagner2000, Krupinski2012, DobrowolskiWagner2020} (see \cite{DobrowolskiWagner2020} for a nice history section of the advances in this approach). Very recently Macpherson and Tent \cite{macpherson2024omegacategoricalpseudofinitegroups} studied pseudofinite $\omega$-categorical groups, with refined conjectures in that case. Here, our approach differs as we will not assume other model theoretic assumptions than $\omega$-categoricity, although we ask for refined algebraic assumptions (e.g. $n$-Engel). As there are no general abstract theory for $\omega$-categorical structures, this paper will not require more model-theoretic notions than that of first-order formulas.

In the long term, our strategy for attacking Wilson's conjecture relies upon the following scheme.
\[\set{\text{ groups }}\underset{\star_1}{\longrightarrow} \set{\text{ Lie algebras }}\underset{\star_2}{\longrightarrow} \set{\text{ associative algebras }}\]
In a broad sense, $\star_1$ is the passage from a group to an associated Lie algebra, sometimes Lie ring (i.e. Lie methods in group theory), and $\star_2$ is the passage from a Lie algebra to an enveloping algebra. The statement above easily translates into similar statements:
\begin{center}
    \textit{Every locally nilpotent $\omega$-categorical Lie}/\textit{associative algebra is nilpotent,}
\end{center}
which we will refer to as (the analog of) Wilson's conjecture for $\omega$-categorical Lie/associative algebras, respectively. In general, $\omega$-categoricity is not preserved under the arrows $\star_1,\star_2$, hence solving the conjecture for e.g. Lie algebras would not solve it for groups (although this happens in some cases, see \ref{section:lazardcorrespondence}).
Nonetheless, our hope is that solving the conjecture on the right side of an arrow can serve as an heuristic for solving it on the left side of the arrow, though at this point, there is no evidence towards thinking that Wilson's conjecture is easier for Lie algebras than for groups.

In two remarkable papers \cite{cherlinnilringI,cherlinnilringII}, Cherlin proves the analog of Wilson's conjecture for associative algebras\footnote{This is a bit 
anachronistic as Wilson's conjecture were stated after Cherlin's result.}, a question asked earlier by Baldwin and Rose \cite{baldwinrose1977}. Recall that a \textit{nilring} is a ring in which for each $a$ there is some $n$ with $a^n = 0$.

\begin{fact}[Cherlin, 1980]\label{fact:cherlinomegacatnilringarenilpotent}
    Any $\omega$-categorical nilring is nilpotent.
\end{fact}
Cherlin's result is in fact stronger, and can be stated in terms of the number of orbits in a particular cartesian power of the algebra. 
\begin{fact}
    Let $A$ be a nilring of bounded nilexponent. 
    \begin{itemize}
        \item If $A$ is commutative and the number of orbits in the action of $\Aut(A)$ on $A\times A\times A$ is finite, then $A$ is nilpotent.
        \item If the number of orbits in the action of $\Aut(A)$ on $A\times A\times A\times A\times A\times A$ is finite, then $A$ is nilpotent.
    \end{itemize}
\end{fact}

The intent of this paper is twofold. First, we review the available literature in order to situate Wilson's conjecture and its analogs for Lie/associative algebras within the current knowledge on locally nilpotent groups and Lie algebras. For instance, Wilson's conjectures for groups and Lie algebras relates to the solution to the restricted Burnside problem and is asymptotically true in a sense that will be made precise later. 

A Lie algebra is \textit{$n$-Engel} if it satisfies the identity 
\[
0 = [\ldots[[a,\underbrace{b],b],\ldots,b}_\text{$n$ times}].
\]
We will see that Wilson's conjecture for Lie algebras reduces to proving that $\omega$-categorical $n$-Engel Lie algebras are nilpotent, for all $n$. A classical result of Higgins \cite{higginsEngel} states that $3$-Engel Lie algebras of characteristic $p\neq 2,5$ are nilpotent. The case of characteristic $5$ is well-known to be pathological and linked to the Burnside problem for groups of exponent $5$. The second goal of this paper is to prove that $\omega$-categorical $3$-Engel Lie algebras over $\F_5$ are nilpotent. 
\begin{theorem*}
    Let $L$ be a $3$-Engel Lie algebra over a field of characteristic $5$. If there are finitely many orbits in the action of $\Aut(L)$ on $L\times L\times L\times L$, then $L$ is nilpotent.
\end{theorem*}
The proof relies on a crucial identity which holds for Lie elements in the enveloping algebra of $3$-Engel Lie algebras of characteristic $5$:
\[
[x,y]^2 = -x^2y^2.
\]
We use this identity to emulate a form of coding that appear in the original proof of Cherlin \cite{cherlinnilringI,cherlinnilringII}. This identity generalizes to the case of $n$-Engel Lie algebras by $[x,y]^{n-1} = \alpha y^{n-1}z^{n-1}$ (for some scalar $\alpha$), and would be quite useful for a general solution to Wilson's conjecture for Lie algebras. As it turns out, this identity is exceptional in $n$-Engel Lie algebras. Using computer algebra, Michael Vaughan-Lee checked that it holds in all $3$-Engel Lie algebras of odd characteristic (if $p\neq 2,5$ those are known to be nilpotent) and for $4$-Engel Lie algebras of characteristic $p>5$ (also known to be nilpotent). However, the identity does not hold for $4$-Engel Lie algebras of characteristic $5$ (not nilpotent in general) and for $n$-Engel Lie algebra for $n>4$.

As we have ruled out the characteristic $5$ case, Wilson's conjecture holds for all $\omega$-categorical $3$-Engel Lie algebras in characteristic $\neq 2$. As a corollary, we recover a partial result for $\omega$-categorical groups. As will be clear later (Fact \ref{fact:fromlitterature}), $\omega$-categorical $p$-groups are \textit{uniformly locally nilpotent}, i.e. for each such group there exists $f:\N\to \N$ such that each $r$-generated subgroup is nilpotent of class at most $f(r)$.

\begin{corollary*}
Every $\omega$-categorical group of exponent $p\geq 5$ in which every $3$-generated subgroup is nilpotent of class $\leq 3$ is nilpotent.
\end{corollary*}

\subsection*{Notations and conventions}
We use the left-normed commutator/bracket convention, which means that in a group or a Lie algebra, $ [\ldots[[a_1,a_2],a_3],\ldots,a_n]$ is denoted by $[a_1,\ldots, a_n]$. We denote by $[a,b^n]$ the term
\[[a,\underbrace{b,\ldots,b}_\text{$n$ times}].\]

\section{Extended introduction}

The goal of this section is to situate Wilson's conjecture (and its Lie/associative algebra analog) within the current literature. The main facts are as follows.
\begin{fact}\label{fact:fromlitterature}
The following are either standard or can be read off the existing literature.
    \begin{enumerate}
        \item An $\omega$-categorical Lie/associative algebra is uniformly locally finite and is over a field of positive characteristic.
        \item Wilson's conjecture for $\omega$-categorical Lie algebras is equivalent to proving that any $\omega$-categorical uniformly locally nilpotent/locally nilpotent/$n$-Engel/Engel Lie algebra is nilpotent.
        \item An $\omega$-categorical $p$-group is uniformly locally finite and of bounded exponent $p^m$. An $\omega$-categorical locally nilpotent group is $n$-Engel, for some $n\in \N$.
        \item \cite{Zelmanov88,Zelmanov1990onproblems} For all $n\in \N$ there exists $N\in \N$ such that for all $n\geq N$, every $n$-Engel $p$-group/Lie algebra over $\F_p$ is nilpotent.
        \item \cite{AbdollahiTraustason2002} Wilson's conjecture is true if for every given $n,p$ and $r$ such that $p^{r-1}< n \leq p^r$, every characteristically simple $\omega$-categorical $n$-Engel group of exponent dividing $p^r$ is abelian.
        \item Under Wilson conjecture for Lie algebras, every $\omega$-categorical $p$-group in which every $3$-generated subgroup is nilpotent of class $<p$ is nilpotent.
    \end{enumerate}
\end{fact}

As this section has an expository flavor, the reader versed in the literature around the Burnside problems and Engel groups is invited to go directly to the next section, where the proof of the main result is given.

\subsection{Uniformity} An $\omega$-categorical locally nilpotent group, Lie algebra or associative algebra is, in fact, \textit{ uniformly locally nilpotent} in the sense that there exists a function $f :\N\to \N$ such that for each $r\in \N$, every $r$-generated substructure is nilpotent of class bounded by $f(r)$. In the language of varieties, the variety generated by the structure is locally nilpotent. To see this, consider a locally nilpotent $\omega$-categorical Lie algebra $L$, fix a number $r$ and for each $k\in \N$ let $S_k$ be the set of tuples $\vec a  = (a_1,\ldots,a_r)\in L^r$ such that $M(a_1,\ldots,a_r) = 0$ for all Lie monomials $M$ of length $\geq k$. By local nilpotency, we have
\[ L ^r = \bigcup_{k\in \N} S_k.\]
Note that each set $S_k$ is $\Aut(L)$-invariant, i.e. is a union of orbits in the action of $\Aut(L)$ on $L^r$. By $\omega$-categoricity, there are only finitely many orbits in this action hence there are only finitely many distinct $S_k$'s, say $S_{k_1},\ldots,S_{k_n}$. By setting $f(r) = \max\set{k_1,\ldots,k_n}$, every $r$-generated Lie subalgebra of $L$ is nilpotent of class $\leq f(r)$. A similar argument yields the analogous result for groups or for associative algebras.

Those kinds of argument often appear in the study of $\omega$-categorical structures. A similar argument yields that an $\omega$-categorical structure is \textit{uniformly locally finite}, that is, there exists $g:\N\to \N$ such that every $r$-generated substructure is of cardinality bounded by $g(r)$. In particular, $\omega$-categorical groups have bounded exponent and $\omega$-categorical Lie and associative algebras are algebras over a field of positive characteristic. Note that as finite $p$-groups are nilpotent, a $p$-group is (uniformly) locally finite if and only if it is (uniformly) locally nilpotent, hence Wilson's conjecture is equivalent to: every $\omega$-categorical $p$-group is nilpotent.

In light of the above, for any locally nilpotent $\omega$-categorical group $G$, there is a number $n$ such that for all $a,b\in G$
\[[a,\underbrace{b,\ldots,b}_\text{$n$ times}] = [a,b^n] = 1.\]
Such groups are called \textit{$n$-Engel groups}. Similarly, an $\omega$-categorical locally nilpotent Lie algebra $L$, is an \textit{$n$-Engel Lie algebra} in the sense that there is $n\in \N$ such that $[a,b^n] = 0$. As for an $\omega$-categorical associative algebra $A$, there is a number $n$ such that $a^n = 0$ for all $a\in A$, i.e. \textit{$A$ is nil of nilexponent $n$}. If the number $n$ depends on $(a,b)$ (respectively $a$), the resulting structure is called an \textit{Engel group}/ \textit{Lie algebra} (resp. \textit{nilalgebra}).

\subsection{Periodic groups and Engel Lie algebras} 

Recall that a group is \textit{periodic} if all its elements have finite order.
In 1902, William Burnside raised his famous problems. 

\begin{generalburnsideproblem}
    Is every finitely generated periodic group finite?
\end{generalburnsideproblem}

\begin{burnsideproblem}
Is every finitely generated group of bounded exponent finite? 
\end{burnsideproblem}

In 1964, Golod and Shafarevich \cite{GolodShafarevitch1964} constructed a counterexample to the general Burnside problem. As for the Burnside problem, it was disproved a few years later in 1968 by Novikov and Adian \cite{NovikovAdianI, NovikovAdianII, NovikovAdianIII}. For $r,n\in \N$, we let $B(r,n)$ denote the quotient of the free group on $r$ generators by the normal closure of the set of $n$-th powers. $B(r,n)$ is known as the \textit{(free) Burnside group of rank $r$ and exponent $n$}. As every $r$-generated group of exponent $n$ is a homomorphic image of $B(r,n)$, the Burnside problem is the question whether $B(r,n)$ is always finite. Nowadays, the problem consist in describing for which $r,n\in \N$ the group $B(r,n)$ is finite. For instance, $B(r,n)$ are known to be finite for $n\leq 6, n\neq 5$, whereas $B(r,n)$ are all infinite for $r\geq 2$ and $n\geq 8000$ \cite{Lysenok1996, Adian1979}. A recent result from Atkarskaya, Rips and Tent \cite{tent2024burnside} is that $B(r,n)$ is infinite for $n$ odd and greater than $557$. The pathological case of exponent $5$ remains open.


\begin{restrictedburnsideproblem}
    For fixed $r,n\in \N$, are there only finitely many finite $r$-generated groups of exponent $n$?
\end{restrictedburnsideproblem}

The restricted Burnside problem asks whether $B(r,n)$ has only finitely many \textit{finite} subgroups. Using a result of Hall and Higman \cite{HallHigman1956} (and the classification of finite simple groups), the restricted Burnside problem reduces to the case of $B(r,p^k)$ which was solved positively in 1959 by Kostrikin \cite{Kostrikin1959} for $B(r,p)$ and by Zelmanov \cite{Zelmanov1990,Zelmanov1991} in 1991 for the general case. Those are a perfect illustration of the use of Lie algebras for solving group-theoretic problems, so we will give some more details of the solution.

There are various ways of associating a Lie ring or a Lie algebra to a given group $G$. Consider the lower central series $(\gamma_i)_{i<\omega}$, for which $\gamma_i/\gamma_{i+1}$ are abelian groups and consider the formal direct product of these groups
\[L(G) := \bigoplus_{i<\omega} \gamma_i/\gamma_{i+1}.\]
The group law of $L(G)$ is abelian, denoted by $+$ and the group commutator $[a,b] = aba^{-1}b^{-1}$ extends partially via the rule $[a\gamma_i, b\gamma_j] = [a,b]\gamma_{i+j}$ and can be linearly extended to the whole of $L(G)$. The structure $(L,+,[.,.])$ thus obtained is a Lie ring. Finite $p$ groups are nilpotent, therefore finding an upper bound to the number of finite subgroups of $B(r,p^k)$ is equivalent to finding an upper bound to the nilpotency class of those finite subgroups. As it turns out, the restricted Burnside problem is equivalent to proving the following:
\begin{center}
    \textit{the associated Lie ring of $B(r,p^k)$ is nilpotent.}
\end{center}

Assume now that $G$ has exponent $p$, then each group $\gamma_i/\gamma_{i+1}$ is an elementary abelian $p$-group and thus $L = L(G)$ is a Lie algebra over $\F_p$. Magnus \cite{Magnus1950} showed that $L$ is $n$-Engel for some $n<p$, i.e. $L$ satisfies the identity $[x,y^n] = 0$. Kostrikin \cite{Kostrikin1959} proved that if $n<p$, every finitely generated $n$-Engel Lie algebra over $\F_p$ is nilpotent, equivalently, every $n$-Engel Lie algebra over $\F_p$ with $n<p$ is locally nilpotent, solving the restricted Burnside problem for groups of exponent $p$. The following problem was raised by Kostrikin \cite{kostrikinAroundBurnside}:
\begin{engelLAproblem}
Is every finitely generated $n$-Engel Lie algebra nilpotent?
\end{engelLAproblem}

For the general solution to the restricted Burnside problem, note that a variant of the construction above using the Zassenhaus central series (or the $p$-central series) allows us to associate a Lie algebra over $\F_p$ to any $p$-group. If the associated Lie ring (or Lie algebra) of $B(r,p^k)$ is not in general an $n$-Engel Lie algebra, it satisfies other (weaker) identities, which Zelmanov proved to be sufficient to obtain local nilpotency, thus solving the restricted Burnside problem. In the end, the results of Kostrikin and Zelmanov are really theorems about Lie algebras, and a byproduct of Zelmanov's proof yields a positive answer to Kostrikin's question.
\begin{theorem*}[Zelmanov, 1991]
    Every $n$-Engel Lie algebra is locally nilpotent. 
\end{theorem*}

An application of compactness to the first order theory of $n$-Engel Lie algebras together with the theorem of Zelmanov (or simply by considering the free $n$-Engel Lie algebra) yields that $n$-Engel Lie algebras are in fact uniformly locally nilpotent. Note that in the case of $\omega$-categorical $n$-Engel Lie algebras, (uniform) local nilpotency follows from local finiteness using Engel's theorem: every finite-dimensional Engel Lie algebra is nilpotent.

\subsection{Global nilpotency of $n$-Engel Lie algebras} The Wilson conjecture is about proving not local but global nilpotency. An earlier and important result of Zelmanov is the following \cite{Zelmanov88}.
\begin{theorem*}[Zelmanov, 1988]
    Every $n$-Engel Lie algebra over a field of characteristic $0$ is nilpotent.
\end{theorem*}
Let $\Delta$ be the two-sorted theory of $n$-Engel Lie algebras over a field of characteristic $0$, i.e. the field sort is axiomatised by the theory of fields and all sentences
\[\theta_p := \underbrace{1+1+\ldots+1}_\text{$p$ times} \neq 0\]
whereas the Lie algebra sort is axiomatised by the theory of Lie algebra with the extra sentence \[\forall x,y\ [x,y^n] = 0.\]
Being nilpotent of class $\leq c$ is expressible by a single sentence $\phi_c$, and Zelmanov theorem can be stated as follows:
\[\Delta \models \bigvee_{c\in \N}\phi_c\]
Compactness implies that a finite fragment of $\Delta$ implies a finite fragment of the right hand side. In particular, only finitely many of the $\theta_p$'s are necessary to imply finitely many of the $\phi_c$'s hence we may conclude the following well-known fact\footnote{See also \cite{VaughanLee2011LiemethodsEngelgroups} for a quick argument for the same result which does not use any logic. Traustason, wrote a constructive account of Zelmanov's proof, yielding explicit bounds for $N,c$, depending on $n$ \cite{TraustasonZelmanovchar01998}. }. 

\begin{fact}[Asymptotic nilpotency for $n$-Engel Lie algebras]
    For each $n\in \N$, there exists $N,c\in \N$ such that every $n$-Engel Lie algebra over a field of characteristic $p>N$ is nilpotent of class $\leq c$.
\end{fact}



For $n\leq 5$, we have a quite satisfactory picture of the situation. In the table below, the number at the intersection of the $n$-Engel row and the characteristic $p$ column is a bound on the nilpotency class of $n$-Engel Lie algebras of characteristic $p$. All bounds are sharp. If that number is $\infty$, we mean that $n$-Engel Lie algebras of characteristic $p$ are not (globally) nilpotent.

\begin{table}[h]
\centering
\begin{tblr}{
  cells = {c},
  vline{1-2,7} = {-}{},
  hline{1-2,6} = {-}{},
}
\diagbox{$n$-Engel Lie algebra~~}{Characteristic $p$~~} & $2$                     & $3$                     & $5$                     & $7$                     & $>7$  \\
$2$-Engel                                                  & $2$                     & $3$                     & $2$ & $2$                     & $2$\\
$3$-Engel                                                  & $\infty$ & $4$                     & $\infty$ & $4$                     & $4$  \\
$4$-Engel                                                  & $\infty$ & $\infty$ & $\infty$ & $7$                     & $7$  \\
$5$-Engel                                                  & $\infty$ & $\infty$ & $\infty$ & $\infty$ & $11$ 
\end{tblr}
\end{table}

The results for $2,3$ and $4$-Engel are due to Higgins \cite{higginsEngel}, Kostrikin \cite{kostrikinAroundBurnside} and Traustason \cite{TraustasonEngel3Engel41993, TraustasonEngel51996, Traustason3-4Engelcharacteristic21995} for exact bounds. The results for $5$-Engel Lie algebras are very recent and due to Vaughan-Lee \cite{vaughanlee20245engelliealgebras}.

Note the following negative result for global nilpotency  of $n$-Engel Lie algebras.
\begin{theorem*}[Razmyslov, 1971]
    For all $p\geq 5$ and $n = p-2$ there exists a non-solvable $n$-Engel Lie algebra over a field of characteristic $p$.
\end{theorem*}

\subsection{Engel groups}

The study of $n$-Engel groups traces back to the 1901 paper of Burnside, where he stated his famous problems. 

\begin{generalengelgroupproblem}
    Is every finitely generated Engel group nilpotent?
\end{generalengelgroupproblem}

\begin{engelgroupproblem}
    Is every finitely generated $n$-Engel group nilpotent?
\end{engelgroupproblem}

If the general local nilpotency problem is known to fail \cite{Traustason2011EngelGroups}, the local nilpotency problem is still open to this day, and is a very active trend of research in group theory. The local nilpotency problem has been answered positively for $n = 2$ \cite{Levi1942}, $n = 3$ \cite{Heineken1961} and $n = 4$ \cite{Traustason19954engelgroups,HavasVaughanLee20054Engelgroups}.

If the local nilpotency problem for Engel Lie algebras is a generalization of the classical Engel's theorem, the local nilpotency problem for Engel groups aims at generalizing Zorn's theorem: \textit{every finite Engel group is nilpotent}. Also, since $\omega$-categorical groups are locally finite, Wilson's conjecture is equivalent to:
\begin{center}
    \textit{every $\omega$-categorical Engel group is nilpotent.}
\end{center}
The ``asymptotic" global nilpotency result for Engel Lie algebras has an equivalent for groups and also stems from Zelmanov's result on global nilpotency of Lie algebras of characteristic $0$ \cite{Zelmanov1990onproblems}.

\begin{theorem*}[Zelmanov, 1990]
    For each $n$ there is a finite set of primes $\pi$ such that any locally nilpotent $n$-Engel group without elements of $\pi$-torsion is nilpotent.
\end{theorem*}

We conclude:

\begin{fact}[Asymptotic nilpotency for $\omega$-categorical $n$-Engel $p$-groups]
    For each $n\in \N$, there exist $N,c\in \N$ such that every $\omega$-categorical $n$-Engel $p$-group with $p>N$ is nilpotent of class $\leq c$.
\end{fact}


As in the case of Lie algebras, some values of $N,c$ are known for particular small values of $n$.

\begin{table}[h]  
\begin{tblr}{
  cells = {c},
  vline{1-2,7} = {-}{},
  hline{1-2,6} = {-}{},
}
\diagbox{$n$-Engel group~~}{$p$-group~~}                       & $2$                     & $3$                     & $5$                     & $7$                     & $>7$  \\
$2$-Engel                                                  & $3$                     & $3$                     & $3$                     & $3$                     & $3$\\
$3$-Engel                                                  & $\infty$                & $4$                     & $\infty$                & $4$                     & $4$  \\
$4$-Engel                                                  & $\infty$                & $\infty$                & $\infty$                & $7$                     & $7$  \\
$5$-Engel                                                  & $\infty$                & $\infty$                & $\infty$                & $\infty$                & $10$ 
\end{tblr}
\end{table}

The results for $2$-Engel groups are due to Hopkins \cite{Hopkins1929}, for $3$-Engel groups to Heineken \cite{Heineken1961} for $4$-Engel groups to Traustason \cite{Traustason19954engelgroups} and Havas and Vaughan-Lee \cite{HavasVaughanLee20054Engelgroups} and for $n = 5$ those results are due to Vaughan-Lee \cite{vaughanlee20245engelliealgebras}. Notice the surprising difference between the group and Lie algebra case for $5$-Engel and $p>7$.  

It should be noted that some of those results are obtained via Lie methods from the corresponding results on Lie algebras, although, as for the associated Lie ring of a group of bounded exponent, the associated Lie ring of an $n$-Engel group is not necessarily an $n$-Engel Lie ring, see \cite{Traustason2011EngelGroups,VaughanLee2011LiemethodsEngelgroups}.

A group is \textit{characteristically simple} if it has no nontrivial characteristic subgroup. An important result of Wilson is:
\begin{fact}[\cite{wilson1981}]
    Wilson's conjecture holds if every characteristically simple $\omega$-categorical $p$-group is abelian.
\end{fact}


\begin{fact}[\cite{AbdollahiTraustason2002}]
    Let $n,p$ be given and $r$ such that $p^{r-1}<n\leq p^r$. Let $G$ be a locally finite $n$-Engel $p$-group.
    \begin{enumerate}
        \item If $p$ is odd, then $G^{p^r}$
is nilpotent of $n$-bounded class.
        \item If $p = 2$ then $(G^{2^r})^2$ is nilpotent of $n$-bounded class.
    \end{enumerate}
\end{fact}

Trivially, a characteristically simple group of exponent $p^k$ satisfies $G^{p^{k-1}} = G$, hence the two previous facts together yield
\begin{fact}
    Wilson's conjecture holds if for every given $n,p$ and $r$ such that $p^{r-1}< n \leq p^r$, every characteristically simple $\omega$-categorical $n$-Engel group of exponent dividing $p^r$ is abelian.
\end{fact}
In particular, for $(p-1)$-Engel $p$-groups, Wilson's conjecture reduces to the case of exponent $p$, for which the associated Lie ring is a $(p-1)$-Engel Lie algebra over $\F_p$.

\subsection{Lazard correspondence}\label{section:lazardcorrespondence} The functor which associates to an $\omega$-categorical $p$-group an associated Lie ring/algebra $L(G)$ does not preserve $\omega$-categoricity in general. Hence, a priori, solving Wilson's conjecture for Lie algebras does not solve Wilson conjecture for groups. However, as any $\omega$-categorical $p$-groups $G$ is uniformly locally nilpotent, there exists $k$ such that every $3$-generated subgroup of $G$ is nilpotent of class $\leq k$. If $k<p$, the local version of the Lazard correspondence allows to define a locally nilpotent Lie algebra structure $L^G$ on $G$ via the (inverse) Baker-Campbell-Hausdorff formula. The operations ``$+$" and ``$[.,.]$" are first-order definable in the structure $G$. As $\omega$-categoricity is a property of the first-order theory of $G$, the Lie algebra $L^G$ is again $\omega$-categorical hence, under Wilson's conjecture for Lie algebra, $L^G$ is nilpotent. By the Baker-Campbell-Hausdorff formula, $G$ is nilpotent. We have proved the following fact.

\begin{fact}(Under Wilson conjecture for Lie algebras)
    Every $\omega$-categorical $p$-group in which every $3$-generated subgroup is nilpotent of class $<p$ is nilpotent.
\end{fact}

\subsection{Associative algebras of bounded nilexponent} The analog of the Burnside problem for associative algebras is the Kurosh-Levitzki problem \cite{Kurosch1941}.

\begin{generalkuroshproblem}
    Is every finitely generated nil algebra nilpotent?
\end{generalkuroshproblem}

\begin{kuroshproblem}
    Is every finitely generated algebra of bounded nilexponent nilpotent?
\end{kuroshproblem}

The counterexample to the general Burnside problem given by the work of Golod and Shavarevitch \cite{Golod1964, GolodShafarevitch1964} is in fact constructed from a counterexample to the general Kurosh-Levitzki problem. On the other hand, the Kurosh-Levitzki problem (unlike the Burnside problem) has a positive solution. It was established by Kaplansky \cite{Kaplansky1946} building on the work of Jacobson \cite{Jacobson1945}. 

\begin{theorem*}[Jacobson, Kaplanski]
    Every associative algebra of bounded nilexponent $n$ is locally nilpotent.
\end{theorem*}

Observe that this result really is an analog to Zelmanov theorem on the local nilpotency of $n$-Engel Lie algebras. Note that in particular, Cherlin's theorem (Fact \ref{fact:cherlinomegacatnilringarenilpotent}) is equivalent to the statement: every locally nilpotent $\omega$-categorical associative algebra is nilpotent. The theorem of Jacobson-Kaplansky was later extended to all PI algebras by Kaplansky \cite{Kaplanski1948} and then even further by Shirshov \cite{Shirshov1957ringswithidentity}.

The question of global nilpotency of associative algebras of bounded nilexponent is solved by the celebrated theorem of Nagata \cite{Nagata1952}, Higman \cite{Higman1956nagata}, Dubnov and Ivanov \cite{DubnovIvanov1943}:
\begin{theorem*}[Nagata, Higman, Dubnov, Ivanov]
    Every associative algebra of bounded nilexponent $n$ over a field of characteristic $0$ or $p>n$ is nilpotent.
\end{theorem*}

\subsection{Conclusion}
To conclude this section, we summarize the situation. We have encountered 4 different local nilpotency problems, in decreasing difficulty: for groups of bounded exponent (the Burnside problem), for $n$-Engel groups, for $n$-Engel Lie algebras and for associative algebras of bounded nilexponent (the Kurosh-Levitzki problem). 

\begin{figure}[h]
    \centering
    
\begin{tabular}{ |c||c|c|c|c|  }
 
 \hline 
 Structure & \multicolumn{2}{c|}{$p$-group}  & Lie algebra over $\F_p$ & Associative algebra over $\F_p$ \\
 \hline
 Bound $n$   & Exponent n    & $n$-Engel &   $n$-Engel & Nilexponent $n$ \\ 
 \hline
 Local nilpotency problem &   False  & Open   & True (Zelmanov) & True (Jacobson-Kaplansky) \\ 
 $p>n$ & \diagbox{}{} & Unknown &  Loc. nilpotent (Kostrikin) & Nilpotent (Nagata-Higman) \\ 
 $p>>n$    & \diagbox{}{} & Nilpotent & Nilpotent & Nilpotent\\ 
 \hline
\end{tabular}
    \label{fig:comparisondifferentproblems}
\end{figure}

As we saw above, by local finiteness of $\omega$-categorical structures, the local nilpotency problem has positive answer for $\omega$-categorical $p$-groups, Engel Lie algebras and associative nilalgebras. The question of Wilson boils down to passing from local nilpotency to global nilpotency.

\section{Solution for $3$-Engel Lie algebras over $\F_5$}
\subsection{Notations} In a Lie algebra $L = (L,+,[.,.])$ over a field $\F$, for $S,R\seq L$ we denote by $[S,R]$ the vector span of $\set{[a,b]\mid a\in L,b\in R}$ and iteratively $[S,R^n] = [R,S,\ldots,S] = [[R,\ldots, S],S]$. We denote by $I(S)$ the ideal in $L$ generated by $S$, i.e. the vector space
\[I(S) = \Span_\F(S)+[S,L]+[S,L,L]+\ldots\]
$I(S)$ is in particular a Lie subalgebra of $L$.
\subsection{Enveloping algebra} Let $L$ be an algebra over a field $\F$. Let $\End(L)$ be the associative algebra of linear endomorphisms of the vector space $L$. For every $a\in L$ the adjoint map $\ad_a : x\mapsto [x,a]$ is an endomorphism of $L$, i.e. an element of $\End(L)$. We define $A(L)$ to be the subalgebra of $\End(L)$ generated by all the endomorphisms $\ad_a$ for $a\in L$. We will usually identify a Lie element in $A(L)$ with its preimage under $\ad$ in $L$.

\subsection{Generalities for $3$-Engel Lie algebras of characteristic $5$}
Recall that a Lie algebra $L$ is $3$-Engel if for all $a,b\in L$ the following equality holds:
\[[a,b^3] = 0.\]

The study of Engel Lie algebras often consists of finding refined identities. The following is classical.
\begin{fact}[Higgins \cite{higginsEngel}]\label{fact:higgins}
Let $L$ be a $3$-Engel Lie algebra over a field of characteristic $\neq 2,3$. Then we have the following identities in $A(L)$ (for all $b,c\in L$)
\begin{enumerate}[label=(\alph*)]
    \item $bc^2+cbc+c^2b = 0$
    \item $3bc^2-3cbc+c^2b = 0$
\end{enumerate}
\end{fact}

For instance, $(a)$ is equivalent to saying that for all $a,b,c\in L$ we have $[a,b,c,c]+[a,c,b,c]+[a,c,c,b] = 0$.

\begin{lemma}\label{lm:crucialidentities}
Let $L$ be a $3$-Engel Lie algebra over a field of characteristic $5$. Then we have the following identities in $A(L)$ (for all $b,c\in L$)
\begin{enumerate}
    \item $bc^2 = c^2b$
    \item $b^2c^2 = c^2b^2$
    \item $[b,c]^2 = -b^2c^2$.
\end{enumerate}
\end{lemma}
\begin{proof}
Subtracting $(a)$ from $(b)$ in Fact \ref{fact:higgins}, we obtain $2bc^2 - 4cbc = 0$ i.e. $bc^2 = 2cbc$. Using $bc^2 = 2cbc$ in the equations $(a)$ we obtain $\frac{3}{2}bc^2 + c^2b = 0$ which, in characteristic $5$ yields $(1)$. We get $(2)$ by associativity using $(1)$: 
\[b^2c^2 = (b^2c)c = (cb^2)c = c(b^2c) = c^2b^2.\]
For $(3)$:
    \begin{align*}
        [b,c]^2 &= (bc-cb)(bc-cb)\\
        &= bcbc-bc^2b-cb^2c+cbcb
    \end{align*}
    By above, $bc^2 = 2cbc$ hence multiplying on the left by $b$ and dividing by $2$, we obtain $\frac{1}{2}b^2c^2 = bcbc$. Exchanging $b$ and $c$ and using $(2)$, yields $\frac{1}{2}b^2c^2 = cbcb$. Using $(1)$ we also have  $bc^2b = cb^2c = b^2c^2$. In turn we conclude 
    \[[b,c]^2 = \frac{1}{2}b^2c^2 -b^2c^2 - b^2c^2 + \frac{1}{2}b^2c^2 = -b^2c^2.\]
\end{proof}

\begin{remark}
$(1)$ and $(2)$ are well-known identities in $3$-Engel Lie algebras of characteristic $\neq 2,3$, but it seems that $(3)$ as stated is new. 
\end{remark}

\subsection{Preparation}
In this subsection, we fix a $3$-Engel Lie algebra $L$ over a field of characteristic $5$. Let $a\in L$ be a fixed element. We define the maps $B_a : L^2\to L$ and $q_a : L\to L$ as follows 
\[B_a(x,y) = [a,x,y]\quad \quad q_a(x) = B(x,x) = [a,x^2].\]
The following is easy to check.
\begin{fact}
    For all $a\in L$, $B_a$ is a bilinear map with associated quadratic map $q_a$. In particular we have for all $x,y$
    \[q_a(x+y) = q_a(x)+q_a(y)+B_a(x,y)+B_a(y,x).\]
\end{fact}

We now define the essential tool for our argument. Let $f_a: L^2\to L$ be defined as
\[f_a(x,y) = q_a([x,y]).\]

\begin{lemma}\label{lm:associativityoff}
    For each $a\in L$ the following properties are satisfied by $f = f_a$ (for all $x,y,z\in L$):
    \begin{enumerate}
        \item $f(x,x) = 0$
        \item $f(x,y) = f(y,x)$
        \item $f([x,y],z) = f(x,[y,z])$.
    \end{enumerate}
\end{lemma}
\begin{proof}
    $(1)$ is trivial since $[x,x] = 0$. $(2)$ holds because $q_a(-x) = q_a(x)$ for all $x$. It remains to check $(3)$. 
    \begin{align*}
        f([x,y],z) = q([[x,y],z]) = [a,[[x,y],z]^2] 
    \end{align*}
    An application of Lemma \ref{lm:crucialidentities} (3) (with $b = [x,y]$ and $c = z$) yields  
    \[[a,[[x,y],z]^2] = -[a,[x,y]^2,z^2]\]
    By Lemma \ref{lm:crucialidentities} (3) we have $[a,[x,y]^2] = - [a,x^2y^2]$ so
    \[-[a,[x,y]^2,z^2] = [a,x^2,y^2,z^2]\]
    It follows that $f([x,y],z) = [a,x^2,y^2,z^2]$
    Note that at this point, by Lemma \ref{lm:crucialidentities} (2), squares of elements of $L$ in $A(L)$ commute hence we already know that 
    \[f([x,y],z) = [a,x^2,y^2,z^2]  = [a,y^2,x^2,z^2] =  [a,x^2,z^2,y^2]\ldots\]
    However this is not needed for $(3)$.
    Using Lemma \ref{lm:crucialidentities} (3) with $a = [a,x^2], b = y^2, c = z^2$ we have 
    \[[a,x^2,y^2,z^2] = [[a,x^2],y^2,z^2] = -[[a,x^2],[y,z]^2] = - [a,x^2,[y,z]^2].\] 
    We conclude using Lemma \ref{lm:crucialidentities} (3) one last time:
    \begin{align*}
        f([x,y],z) &= [a,x^2,y^2,z^2]\\
        &= - [a,x^2,[y, z]^2]\\
        &= [a[x,[yz]]^2]\\
        &= f(x,[y,z])
    \end{align*}
\end{proof}

We can now freely use the notation 
\[f(x_1,\ldots,x_n) = q([x_1,\ldots ,x_n]).\]

\begin{lemma}\label{lm:transitivityofequalitythroughf}
    For all $x_1,\ldots,x_n\in L$ and for all $\sigma\in \mathfrak S _n$ we have 
    \[f(x_1,\ldots,x_n) = [a[x_1\ldots x_n]^2] = (-1)^{n+1}[ax_1^2\ldots x_n^2] = f(x_{\sigma(1)},\ldots, x_{\sigma(n)}).\]
    In particular, for all $x,y,x',y',z\in L$
    \begin{itemize}
        \item $[f(x,y),z^2] = -f(x,y,z)$,
        \item if $f(x,y) = f(x',y')$ then $f(x,y,z) = f(x',y',z)$ and $f(z,x,y) = f(z,x',y')$,
        \item if $f(x,y) = 0$ then $f(x,y,z) = 0$.
    \end{itemize}
\end{lemma}
\begin{proof}
    From an easy induction using Lemma \ref{lm:crucialidentities} (3), we obtain
    \[[a[x_1,\ldots,x_n]^2] = [a[[x_1,\ldots,x_{n-1}]x_n]^2]  = -[a[x_1,\ldots,x_{n-1}]^2x_n^2] = \ldots = (-1)^{n+1}[ax_1^2\ldots x_n^2].\]
    Then from Lemma \ref{lm:crucialidentities} (2), the maps $x\mapsto [x,x_i^2]$ and $x\mapsto [x,x_j^2]$ commute for all $i,j$ so we get the first series of equalities. For the first item, 
    \[[f(x,y),z^2] = [-[a,x^2,y^2],z^2] = -[a,x^2,y^2,z^2] = -f(x,y,z).\]
    The last two items follow immediately from the first item.
\end{proof}


\subsection{First-order expressibility} We now turn to first-order logic. We work in the language $\LL = \set{+,[.,.],0}$ of Lie rings. Given $k\in \N^{>1}$, we define two formulas
\[\phi_k(z_0,z_1,z_2,z_3) = \exists y_1,\ldots,y_{k+1} \ f_{z_0}(z_2,z_3,y_1)\neq 0 \wedge \bigwedge_{i = 1}^k f_{z_0}(z_1,y_i) = f_{z_0}(z_2,y_{i+1})\wedge f_{z_0}(z_1,y_{k+1}) = 0,\]

\[\psi_k(z_0,z_1,z_2,z_3) = \exists x_1,\ldots,x_{k}\ (f_{z_0}(z_2,z_3) = f_{z_0}(x_1,z_1) \wedge \bigwedge_{i = 1}^{k-1} f_{z_0}(x_i,z_2) = f_{z_0}(x_{i+1},z_1).\]

\begin{lemma}\label{lm:3engel-inconsistent}
    If $L$ is a $3$-Engel Lie algebra over $\F_5$ and $l,k\in \N^{>1}$ are such that $l> k$, then $(\phi_k\wedge \psi_l)$ has no realisations in $L$. 
\end{lemma}
\begin{proof}
Assume $L\models (\phi_k\wedge \psi_l)(A_0,\ldots,A_3)$ for some for some $A_0,\ldots,A_3\in L$. Let $y_1,\ldots,y_{k+1}$ be the witnesses of $\phi_k(A_0,\ldots,A_3)$. As $l>k$ we have $x_1,\ldots,x_{k+1},\ldots,x_l$ witnessing $\phi_l(A_0,\ldots,A_3)$. We write $f = f_{A_0}$. We have $0\neq f(A_2,A_3,y_1)$. As $f(A_2,A_3) = f(x_1,A_1)$ we may use Lemma \ref{lm:transitivityofequalitythroughf} to get $0\neq f(x_1,A_1,y_1)$. As $f(A_1,y_1) = f(A_2,y_2)$ we may again use Lemma \ref{lm:transitivityofequalitythroughf} to get $0\neq f(x_1,A_2,y_2)$. Iterating the use of Lemma \ref{lm:transitivityofequalitythroughf}, we get $0\neq f(x_{k},A_2,y_{k+1}) = f(x_{k+1},A_1,y_{k+1})$ which contradicts $f(A_1,y_{k+1}) = 0$ by Lemma \ref{lm:transitivityofequalitythroughf}.
\end{proof}

For the next lemma, we will use the following fact of Traustason.
\begin{fact}[Traustason \cite{TraustasonEngel3Engel41993}]\label{fact:traustason}
Let $L$ be a $3$-Engel Lie algebra over a field of characteristic $\neq 2,3$, then for all $a\in L$, the ideal $I(a)$ spanned by $a$ is nilpotent of class $<3$, i.e.
\[[I(a),I(a),I(a)] = 0.\]
\end{fact}

\begin{lemma}\label{lm:3engel-consistent}
    If $L$ is a $3$-Engel Lie algebra over $\F_5$. Let $k\in \N^{>1}$ and assume that for $n = 2k+1$, there exists $a,b_1,\ldots,b_{n}$ such that
    \[0\neq [a,b_1^2,\ldots ,b_n^2].\]
    Then $(\phi_k\wedge \psi_k)$ has a realisation in $L$.
\end{lemma}
\begin{proof}
    We define the following
    \begin{align*}
        A_0 &= a\\
        A_1 &= \sum_{i = 2}^k [b_i,b_{k+i}] + [b_1,b_{k+1},b_{2k+1}]\\
        A_2 &= \sum_{i=1}^{k-1} [b_i,b_{k+i+1}]+[b_k,b_{k+1}]\\
        A_3 &= [b_1,\ldots ,b_{k-1}, b_{2k+1}]
    \end{align*}
    As above, we will write $f$ for $f_{A_0} = f_a$. Similarly, $q = q_a$ and $B = B_a$.

    \noindent \textbf{Step 1.} We prove that $L\models \phi_k(A_0,\ldots,A_3)$. We define 
    \[y_1 := [b_{k+2},\ldots ,b_{2k}] \ ;\  y_i = [b_{k+1},\ldots ,\hat{b_{k+i}}, \ldots ,b_{2k}]\ (1<i\leq k)\ ; \ y_{k+1} = [b_{k+1},\ldots ,b_{2k+1}].\] 
    We have to check that $f(A_2,A_3,y_1)\neq 0$, $f(A_1,y_i) = f(A_2,y_{i+1})$ and $f(A_1,y_{k+1}) = 0$. We set up some conventions and notations.

The \textit{weight} of a Lie monomial $M$ in $x$ is the number of times $x$ occurs in $M$.
    \begin{claim}\label{claim_bilinearmapequals0}
        Assume that $M,N$ are two Lie monomials such $M$ is of weight $\geq 2$ in some $b\in L$ and $N$, is of weight $1$ in $b$ then 
        \[q(M) = B(M,N) = B(N,M) =  0.\]
    \end{claim}
    \begin{proof}
        As $q(M) = [a,M^2]$, the weight of $q(M)$ in $b$ is $\geq 4$ hence $q(M) = 0$ by Fact \ref{fact:traustason}. Similarly the weight of $B(M,N) = [a,M,N]$ in $b$ is $\geq 3$ hence $B(M,N) = 0$.
    \end{proof}
    By Fact \ref{fact:traustason}, any nontrivial Lie monomial in $L$ has weight $0,1,2$ in any given element of $L$.

    Given a multiset $\set{x_1,\ldots,x_n}$ of elements of $L$, we will usually denote by $M\set{x_1,\ldots,x_n}$ a Lie monomial using each element of the multiset once and only once. For instance $[x,y^2,z,y] = M\set{x,y,y,y,z}$ and $[x,[y,[x,y]]] = M\set{x,x,y,y}$. 

    We prove that $f(A_2,A_3,y_1)\neq 0$. We can then compute
    \begin{itemize}
        \item $[A_2,A_3] = \sum_{i = 1}^{k-1} M\set{b_i,b_{k+i+1},b_1,\ldots,b_{k-1},b_{2k+1}}+M\set{b_1,\ldots,b_{k+1},b_{2k+1}}$\hfill $(\star)_1$
        \item $[A_2,A_3,y_1] = \sum_{i = 1}^{k-1}M\set{b_ib_{k+i+1},b_1,\ldots,b_{k-1},b_{k+2}\ldots b_{2k+1}}+M\set{b_1,\ldots,b_{2k+1}}$
    \end{itemize}
    For $1\leq i\leq k-1$, we set 
    \[M_i = M\set{b_ib_{k+i+1},b_1,\ldots,b_{k-1},b_{k+2}\ldots b_{2k+1}}\]  and $M_0 = M\set{b_1,\ldots,b_{2k+1}}$.
    We obtain:
    \[f(A_2,A_3,y_1) = q([A_2,A_3,y_1]) = q(\sum_{i = 0}^{k-1} M_i) = \sum_{i,j = 0}^{k-1}B(M_i,M_j)\]
    We have that $M_0$ is of weight $1$ in $b_i$ for each $1\leq i\leq 2k+1$ (hence of weight $2$ in none of the $b_i$'s). Further, $M_i$ is of weight $2$ in $b_i$ and of weight $1$ in $b_j$ for $1\leq i\neq j\leq k-1$. We conclude that all terms $B(M_i,M_j)$ vanish except for $i = j = 0$, so $f(A_2,A_3,y_1) = q(M_0)$. By Lemma \ref{lm:transitivityofequalitythroughf} we also have 
    \[q(M\set{b_1,\ldots,b_{2k+1}}) = q([b_1,\ldots ,b_{2k+1}]) = [a,b_1^2,\ldots ,b_n^2]\]
    so $f(A_2,A_3,y_1)\neq 0$ by hypothesis.

    We check that $f(A_1,y_i) = f(A_2,y_{i+1})$ for all $1\leq i\leq k$. Fix $1\leq i_0\leq k$. We compute
    \begin{itemize}
        \item $[A_1,y_{i_0}] = \sum_{i = 2}^k M\set{b_i,b_{k+i},b_{k+1},\ldots,\hat{b_{k+i_0}},\ldots b_{2k}} + M\set{b_1,b_{k+1},b_{k+1},\ldots, \hat{b_{i_0}},\ldots,b_{2k+1}}$
        \item $[A_2y_{i_0+1}] = \sum_{i = 1}^{k-1} M\set{b_i,b_{k+i+1},b_{k+1},\ldots,\hat{b_{k+i_0+1}},\ldots,b_{2k}} + M\set{b_k,b_{k+1},b_{k+1},\ldots,\hat{b_{k+i_0+1}},\ldots,b_{2k}}$
    \end{itemize}
    We treat each equality separately. Set $M_0 = M\set{b_1,b_{k+1},b_{k+1},\ldots, \hat{b_{i_0}},\ldots,b_{2k+1}}$ and for $2\leq i\leq k$, set
    \[M_i = M\set{b_i,b_{k+i},b_{k+1},\ldots,\hat{b_{k+i_0}},\ldots b_{2k}}.\] Then $M_0$ is of weight $2$ in $b_{k+1}$ and $M_i$ is of weight $1$ in $b_{k+1}$ so $q(M_0) = B(M_0,M_i) = 0$ for all $2\leq i\leq k$. Further, for $2\leq i\neq j\leq k$, $M_i$ is of weight $2$ in $k+i$ and of weight $1$ in $k+j$ except $M_{i_0} = M\set{b_{i_0},b_{k+1},\ldots,b_{2k}}$ which is of weight $1$ in each entries. We conclude that for $2\leq i,j\neq i_0$ we have $B(M_i,M_j) = B(M_i,M_{i_0}) = 0$ and hence
    \[f(A_1,y_{i_0}) = q([A_1,y_{i_0}]) = \sum_{i,j = 0,2,\ldots}^kB(M_i,M_j) = B(M_{i_0},M_{i_0}) = f(b_{i_0},b_{k+1},\ldots,b_{2k})
    \]
    Now set $N_0 = M\set{b_k,b_{k+1},b_{k+1},\ldots,\hat{b_{k+i_0+1}},\ldots,b_{2k}}$ and $N_i = M\set{b_i,b_{k+i+1},b_{k+1},\ldots,\hat{b_{k+i_0+1}},\ldots,b_{2k}}$ for $1\leq i\leq k-1$. As $N_0$ is of weight $2$ in $b_{k+1}$ and each $N_i$ is of weight $1$ in $b_{k+1}$ we have that $B(N_0,N_i) = 0$ for all $1\leq i\leq k-1$. Then, for all $1\leq i\neq j\leq k-1$, each $N_i$ is of weight one in $k+j$ and of weight $2$ in $b_{k+i+1}$ except $N_{i_0}$ which is of weight $1$ in all of its entries. We conclude
    \[f(A_2,y_{i_0+1}) = \sum_{i,j = 0}^k B(N_i,N_j) = q(N_{i_0}) = f(b_{i_0},b_{k+1},\ldots,b_{2k}).\]
    We conclude $f(A_1,y_{i_0}) = f(A_2,y_{i_0+1})$.

    It remains to check that $f(A_1,y_{k+1}) = 0$. We compute
    \[[A_1,y_{k+1}] = \sum_{i = 2}^k M\set{b_ib_{k+i}b_{k+1},\ldots,b_{2k+1}}+M\set{b_1,b_{k+1},b_{k+1},b_{k+2}\ldots,b_{2k},b_{2k+1}b_{2k+1}}\]
    For $2\leq i\leq k$ set $M_i = M\set{b_i,b_{k+i},b_{k+1},\ldots,b_{2k+1}}$ and set $M_0 = M\set{b_1,b_{k+1},b_{k+1},b_{k+2}\ldots,b_{2k},b_{2k+1}b_{2k+1}}$. Then $M_0$ has weight $2$ in $b_{k+1}$ and each $M_i$ has weight $1$ in $b_{k+1}$ so $B(M_0,M_i) = B(M_i,M_0) = 0$ for all $2\leq i\leq k$. For each $2\leq i\neq j\leq k$, $M_i$ has weight $2$ in $b_{k+i}$ and $M_j$ has weight one in $b_{k+i}$ so $B(M_i,M_j) = 0$. We conclude that $f(A_1,y_{k+1}) = 0$.

    \noindent \textbf{Step 2.} We prove that $L\models \psi_k(A_0,\ldots,A_3)$. For each $1\leq i\leq k$ we set 
    \[x_i = [b_1,\ldots ,\hat{b_i},\ldots ,b_k]. \]

    First we check that $f(A_2,A_3) = f(x_1,A_1)$. From $(\star)_1$ and by computing, we have 
    \begin{itemize}
        \item $[A_2,A_3] = \sum_{i = 1}^{k-1} M\set{b_i,b_{k+i+1},b_1,\ldots,b_{k-1},b_{2k+1}}+M\set{b_1,\ldots,b_{k+1},b_{2k+1}}$
        \item $[x_1,A_1] = \sum_{i = 2}^kM\set{b_i,b_{k+i},b_2,\ldots,b_k}+ M\set{b_1,\ldots,b_{k+1},b_{2k+1}}$
    \end{itemize}

By computing each term separately and sorting out the weight of each monomial, we get 
\[f(A_2,A_3) = f(b_1,\ldots,b_{k+1},b_{2k+1}) = f(x_1,A_1).\]
Similary, for each $1\leq i_0\leq k-1$ we get 
\begin{itemize}
    \item $[x_{i_0},A_2] = \sum_{i = 1}^{k-1} M\set{b_1,\ldots,\hat{b_{i_0}},\ldots,b_k,b_i,b_{k+i+1}}+M\set{b_1,\ldots,\hat{b_{i_0}},\ldots,b_k,b_k,b_{k+1}}$
    \item $[x_{i_0+1},A_1] = \sum_{i=2}^k M\set{b_1,\ldots,\hat{b_{i_0+1}},\ldots,b_k,b_i,b_{k+i}}+M\set{b_1,\ldots,\hat{b_{i_0+1}},\ldots,b_k,b_1,b_{k+1},b_{2k+1}}$.
\end{itemize}
By sorting out the weights and reasoning as above, we get 

\[f(x_{i_0},A_2) = f(b_1,\ldots,b_k,b_{k+i_0+1}) = f(x_{i_0+1},A_1).\]
This concludes the proof of the lemma.
\end{proof}

\subsection{Conclusion}

\begin{theorem}\label{thm:finite4typesEngel3char5nilpotent}
    Let $L$ be a $3$-Engel Lie algebra over a field of characteristic $5$. Assume that that there are finitely many orbits in the action of $\Aut(L)$ on $L\times L\times L\times L$, then $L$ is nilpotent.
\end{theorem}
\begin{proof}
    Let $L$ be as in the hypotheses. We start with a claim.
    \begin{claim}
        There exists $n\in \N$ such that $[a,b_1^2,\ldots ,b_n^2] = 0$ for all $a,b_1,\ldots,b_n\in L$.
    \end{claim}
    \begin{proof}
        Assume not, then by Lemma \ref{lm:3engel-consistent} for each $k\in \N$ there exists a $4$ tuple $\vec c_k = (c_{0,k},\ldots,c_{3,k})$ which satisfies $\Phi_k := (\phi_k\wedge \psi_k)$ in $L$. If 
        \[ P_k := \set{(u_0,\ldots, u_3)\in L^4\mid L\models \Phi_k(u_0,\ldots,u_3)}\]
        then $\vec c_k\in P_k$, for each $k$. Observe that $P_k$ and hence $\bigcup_{k\in \N} P_k$ are closed under the action of $\Aut(L)$. As $L\times L\times L\times L$ has only finitely many orbits under $\Aut(L)$, there exists $k\neq l\in \N$ and $\sigma\in \Aut(L)$ such that $\sigma(\vec c_k) = \vec c_l$. This implies that $c_k\in P_k\cap P_l$, which contradicts Lemma \ref{lm:3engel-inconsistent}. 
    \end{proof}

     Let $I = \Span_\F([a,b^2]\mid a,b\in L)$. By Lemma \ref{lm:crucialidentities} (1), for all $c\in L$ we have $[[a,b^2],c] = [[a,c],b^2]$ hence it follows that $I$ is an ideal of $L$. The quotient Lie algebra $L/I$ is $2$-Engel hence nilpotent of class $2$. It follows that the product of every $3$ elements in $L$ is of the form $\sum_i[a_i,b_i^2]$. Let $c = 3+2(n-1)$,  $x_1,\ldots,x_{c}\in L$ and write $[x_1, x_2 ,x_3] = \sum_i [a_i,b_i^2]$. Then 
     \[[x_1,\ldots ,x_{c}] = \sum_i[ [a_i,b_i^2],x_4,\ldots ,x_{c}] = \sum_i[ a_i,x_4,\ldots ,x_{c+1},b_i^2]\]
     By immediate iteration, the right hand side is a sum of element of the form $[a,b_1^2,\ldots ,b_n^2]$ hence $[x_1,\ldots ,x_{c}] = 0$.
\end{proof}

\begin{remark}
Note that the identity
    $[a_1,b_1^2,a_2,b_2^2,\ldots ,a_n,b_n^2] = 0$ follows from the identity $[a_1,b_1^2,\ldots ,b_n^2] = 0$ using iteratively Lemma \ref{lm:crucialidentities}. It follows that not only $L/I$ is nilpotent but also $I$. One can use a classical result of Higgins to conclude that $L$ is nilpotent.
\end{remark}

We deduce a very particular case of Wilson's conjecture for groups.
\begin{corollary}
    Let $G$ be an $\omega$-categorical group of prime exponent $p>3$ such that every $3$-generated subgroup has nilpotency class $\leq 3$. Then $G$ is nilpotent.
\end{corollary}

\begin{proof}
    We invoke the Lazard correspondence. As the nilpotency class of every $3$-generated subgroup is less than $p$, we use the inverse Baker-Campbell-Hausdorff formula in order to define an $\F_p$-Lie algebra structure $(L,+,[.,.])$ with domain $G$. Then $L$ is again $\omega$-categorical and as the nilpotency class of every two generated subgroups is $\leq 3$, $L$ is $3$-Engel hence nilpotent by above.
\end{proof}


\subsection*{Acknowledgement} 
Je suis éternellement reconnaissant envers Ombeline Fougerat pour son support constant et infaillible. I am immensely grateful to Dugald Macpherson for suggesting this fascinating problem and for numerous discussions and counsel. I am indebted to Aris Papadopoulos for hearing my constant blabbering about Engel Lie algebras during the last few months. I am very thankful to Gunnar Traustason for providing me his PhD thesis, crucial for this work and answering many of my questions. I am very grateful to Michael Vaughan-Lee for his availability to answer my many questions. I am thankful to Yves de Cornulier for discussions on the Lazard correspondence.

\bibliographystyle{plain}
\bibliography{biblio.bib}{}

\end{document}